\documentclass[letter,final]{article}
\usepackage[english]{babel}
\usepackage{amsmath,amsfonts,amssymb,amsthm}
\usepackage{latexsym,stmaryrd}
\usepackage{enumerate}

\newcommand{\N}{\mathbb{N}}
\newcommand{\Z}{\mathbb{Z}}
\newcommand{\Q}{\mathbb{Q}}
\newcommand{\R}{\mathbb{R}}
\newcommand{\C}{\mathbb{C}}
\newcommand{\F}{\mathbb{F}}

\newcommand{\calO}{\mathcal{O}}
\newcommand{\fraka}{\mathfrak{a}}
\newcommand{\frakm}{\mathfrak{m}}
\newcommand{\frakp}{\mathfrak{p}}
\newcommand{\frakq}{\mathfrak{q}}

\newcommand{\abs}[1]{{\left|{#1}\right|}}
\newcommand{\floor}[1]{{\left\lfloor{#1}\right\rfloor}}

\DeclareMathOperator{\Div}{Div}
\DeclareMathOperator{\Princ}{Princ}
\DeclareMathOperator{\Pic}{Pic}
\DeclareMathOperator{\Id}{Id}
\DeclareMathOperator{\PId}{PId}
\DeclareMathOperator{\support}{supp}

\DeclareMathOperator{\Conorm}{Con}

{\catcode`\!=8 
 \catcode`\Q=3 
\long\gdef\given#1{88\fi\Ifbl@nk#1QQQ\empty!}
\long\gdef\blank#1{88\fi\Ifbl@nk#1QQ..!}
\long\gdef\nil#1{\IfN@Ught#1* {#1}!}
\long\gdef\IfN@Ught#1 #2!{\blank{#2}}
\long\gdef\Ifbl@nk#1#2Q#3!{\ifx#3}
}

\newenvironment{faketheorem}[2] 
{\trivlist \item[]{\bfseries #1}#2}
{\endtrivlist}
\newcommand{\newfaketheorem}[4]
{\newenvironment{#1}[1][]%
{\refstepcounter{#3}%
\begin{faketheorem}{#2\ \csname the#3\endcsname\if\blank{##1}\else\ (##1)\fi.}%
  {\if\blank{#4}\itshape\else#4\fi}\ }
  {\end{faketheorem}}}
\newcommand{\newfaketheoremast}[3]
{\newenvironment{#1}[1][]%
{\begin{faketheorem}{#2\if\blank{##1}\else\ (##1)\fi.}{\if\blank{#3}\itshape\else#3\fi}\ }
  {\end{faketheorem}}}

\newcounter{definition}[section]

\newfaketheorem{definition}{Definition}{definition}{\itshape}
\newfaketheorem{theorem}{Theorem}{definition}{\itshape}
\newfaketheorem{proposition}{Proposition}{definition}{\itshape}
\newfaketheorem{lemma}{Lemma}{definition}{\itshape}
\newfaketheorem{corollary}{Corollary}{definition}{\itshape}
\newfaketheorem{remark}{Remark}{definition}{\upshape}
\newfaketheorem{remarks}{Remarks}{definition}{\upshape}
\newfaketheorem{example}{Example}{definition}{\upshape}
\newfaketheorem{examples}{Examples}{definition}{\upshape}
\newfaketheorem{conjecture}{Conjecture}{definition}{\itshape}

\newfaketheoremast{definition*}{Definition}{\itshape}
\newfaketheoremast{theorem*}{Theorem}{\itshape}
\newfaketheoremast{proposition*}{Proposition}{\itshape}
\newfaketheoremast{lemma*}{Lemma}{\itshape}
\newfaketheoremast{corollary*}{Corollary}{\itshape}
\newfaketheoremast{remark*}{Remark}{\upshape}
\newfaketheoremast{remarks*}{Remarks}{\upshape}
\newfaketheoremast{example*}{Example}{\upshape}
\newfaketheoremast{examples*}{Examples}{\upshape}

\newenvironment{enuma}{\begin{enumerate}[\upshape (a)]}{\end{enumerate}}
\newenvironment{enumi}{\begin{enumerate}[\upshape (i)]}{\end{enumerate}}

\title{Holes in the Infrastructure of Global Hyperelliptic Function Fields}
\author{Felix Fontein\footnote{Department of Mathematics \&\ Statistics, University of Calgary, 2500
University Drive NW, Calgary, Alberta, Canada T2N 1N4. Email: fwfontei@ucalgary.ca}}

\DeclareMathOperator{\Red}{Red}
\DeclareMathOperator{\ideal}{ideal}
\DeclareMathOperator{\signature}{sig}

\begin{document}
  \maketitle
  
  \begin{abstract}
    We prove that the number of ``hole elements''~$H(K)$ in the infrastructure of a hyperelliptic
    function field~$K$ of genus~$g$ with finite constant field~$\F_q$ with $n + 1$ places at
    infinity, of whom $n' + 1$ are of degree~one, satisfies \[ \abs{\frac{H(K)}{\abs{\Pic^0(K)}} -
    \frac{n'}{q}} = \calO(16^g n q^{-3/2}). \] We obtain an explicit formula for the number of holes
    using only information on the infinite places and the coefficients of the $L$-polynomial of the
    hyperelliptic function field. This proves a special case of a conjecture by E.~Landquist and the
    author on the number of holes of an infrastructure of a global function field.
    
    Moreover, we investigate the size of a hole in case $n = n'$, and show that asymptotically for
    $n \to \infty$, the size of a hole next to a reduced divisor~$D$ behaves like the function
    $\frac{n^{g - \deg D}}{(g - \deg D)!}$.
  \end{abstract}
  
  \section{Introduction}
  
  When considering the infrastructure of a global function field~$K / \F_q(x)$, it turns out that
  its set of $f$-representations represents the divisor class group~$\Pic^0(K)$ (see
  \cite{ff-tioagfoaur}). Some elements of $\Pic^0(K)$ are directly represented by reduced ideals of
  $\calO$, the integral closure of $\F_q[x]$ in $K$, while others need additional information for
  the infinite places. We say that an element of $\Pic^0(K)$ is a \emph{hole element} if it does not
  corresponds to a reduced ideal. (Note that we will make this more precise in
  Section~\ref{relatinginfrastructure}.) Two natural questions is: how many hole elements are there?
  And if some of them cluster together, what is the size of this cluster?
  
  These questions are related to applications. When implementing arithmetic in $\Pic^0(K)$ using the
  infrastructure, it is most efficient if one avoids hole elements. Hence, if the number of hole
  elements is small, the chance that one avoids hole elements is big: it makes sense to optimize the
  algorithms for arithmetic for this case.
  
  In practical experiments made by E.~Landquist \cite{landquist-thesis} and the author
  \cite{fontein-diss} with a small number of infinite places, it turns out that the chance that a
  random element of $\Pic^0(K)$ is a hole element is about $\frac{1}{q}$. There is also a heuristic
  explanation of this phenomenon: hole elements only occur near to reduced ideals of degree~$<
  g$. Assuming that the norms of reduced ideals are uniformly distributed in $(\F_q)_{\le g}[x]$,
  one obtains that the chance that a random reduced ideal has degree~$< g$ should be around
  $\frac{1}{q}$ \cite[p.~132]{fontein-diss}.
  
  In this paper, we will show that this conjecture is (almost) true for hyperelliptic\footnote{We
  consider elliptic function fields as a special case of hyperelliptic function fields.} function
  fields~$K$: the probability is not $\frac{1}{q}$ but $\frac{n}{q}$, if $n' + 1$ is the number of
  infinite places of $K / \F_q(x)$ which have degree~one. The error term is indeed dominated by
  $q^{-3/2}$ in this case, with an additional factor of $16^g n$, where $g$ is the genus of $K$ and
  $n + 1$ the total number of infinite places of $K / \F_q(x)$.
  
  Another question related to holes of infrastructures is their ``size'': hole elements group
  ``around'' reduced ideals of degree~$< g$. When we define a \emph{hole} next to a reduced
  ideal~$\fraka$ as the set of hole elements whose reduced ideal is $\fraka$, one can give upper and
  lower bounds for the size of a hole in case $n = n'$, i.e. all infinite places of $K / \F_q(x)$
  are of degree~one. It turns out that asymptotically for $n \to \infty$, assuming that $\fraka$
  avoids certain places, our bounds imply that the size behaves like the function $\frac{n^{g - \deg
  \fraka}}{(g - \deg \fraka)!}$.
  
  We first investigate the arithmetic of hyperelliptic function fields in
  Section~\ref{arithhypersection}, to get an explicit description of the set of reduced divisors. In
  Section~\ref{relatinginfrastructure}, we sketch how the infrastructure looks like and make more
  precise what the holes in it are. Moreover, we consider the elliptic function field case and give
  a (mostly) ``local'' criterion whether a divisor is reduced. By generalizing this definition of
  being reduced, we describe generating functions of these sets of reduced divisors vanishing at a
  set~$S$ of places\footnote{In the above notation, $\abs{S} = n + 1$.} in Section~\ref{genfuns}. We
  investigate how these functions change when $S$ changes and obtain an explicit description of the
  set of holes. Next, we consider the case of $\abs{S} = 1$ in Section~\ref{countingreddivs}, show
  that the generating function is rational and give estimations for certain coefficients of the
  generating function. In Section~\ref{holecountingsect}, we use the results to show our first main
  result, namely a bound on the number of holes and an explicit formula, and in
  Section~\ref{onsizeofholes} we show the result on the size of holes. Finally, we conclude in
  Section~\ref{conclusion} with a few conjectures for the case of arbitrary global function fields.
  
  \subsection*{Acknowledgments}
  I would like to thank Eric Landquist for the discussions on the number of holes and Renate
  Scheidler for the suggestion to study the hyperelliptic case. Moreover, I would like to thank
  Florian He\ss\ and Andreas Stein for their comments after a talk the author gave on this subject
  at the Carl von Ossietzky University of Oldenburg.
  
  \section{Arithmetic of Hyperelliptic Function Fields}
  \label{arithhypersection}
  
  We begin with reviewing the arithmetic of a hyperelliptic function field~$K$. We want to work with
  a representation of $K/R$, where $R$ is a quadratic rational subfield of $K$, without having the
  notion of an infinite place of $R$. We begin with very general results on the arithmetic of a
  function field.
  
  Let $k$ be a perfect field and $K$ a hyperelliptic function field with exact field of constants
  $k$. Let $g$ denote the genus of $K$. If $g > 1$, let $R$ be the unique rational subfield of $K$
  with $[K : R] = 2$ and let $\Conorm_{K/R}$ denote the conorm map $\Div(R) \to \Div(K)$.
  
  Let $\frakp$ be a place of $K$ of degree~one. The first few results are true for arbitrary
  function fields as well: the only requirement is that $\frakp$ is a place of degree~one. For more
  information, see \cite{hessRR, ff-tioagfoaur}.
  
  \begin{definition}
    Let $D \in \Div(K)$ be a divisor. We say that $D$ is \emph{reduced} with respect to $\frakp$ if
    $L(D) = k$ and $\nu_\frakp(D) = 0$. Denote the set of reduced divisors by $\Red_\frakp(K)$.
  \end{definition}
  
  Note that we always have $0 \in \Red_\frakp(K)$. Reduced divisors allow to describe the divisor
  class group~$\Pic^0(K) = \Div^0(K) / \Princ(K)$:
  
  
    
    
  
  \begin{proposition}
    The map \[ \Red_\frakp(K) \to \Pic^0(K), \qquad D \mapsto D - (\deg D) \frakp + \Princ(K) \] is
    a bijection, mapping $0 \in \Red_\frakp(K)$ to the neutral element of $\Pic^0(K)$. \qed
  \end{proposition}
  
  One can state a few properties on reduced divisors; we will see that in the hyperelliptic case,
  some of these properties already characterize reduced divisors in hyperelliptic function fields:
  
  \begin{lemma}
    Let $D \in \Red_\frakp(K)$. Then $D \ge 0$ and $\deg D \le g$. In case $R$ is any rational
    subfield of $K$, and we have $D' \in \Div(R)$ with $\Conorm_{K/R}(D') \le D$, then $D' \le 0$.
  \end{lemma}
  
  In particular, if $K$ is a rational function field, this shows that $\Red_\frakp(R) = \{ 0 \}$,
  i.e. $\Pic^0(K) = 0$.
  
  \begin{proof}
    First, as $1 \in L(D)$, we must have $D \ge 0$. Next, if $\deg D > g$, we have $\dim L(D) \ge
    \deg D + 1 - g > 1$ by Riemann's Inequality, contradicting $L(D) = k$.
    
    Now assume that $R$ is any rational subfield of $K$. Let $\frakq = \frakp \cap R$ and let $D'
    \in \Div(R)$ with $\Conorm_{K/R}(D') \le D$. Without loss of generality, we can assume $D' \ge
    0$. Now $\deg \frakq = 1$, whence $D'' := D' - (\deg D') \frakq$ is principal; i.e. there exists
    some $x \in R^*$ with $(x)_R = D''$. Now
    \begin{align*}
      \Conorm_{K/R}(D') ={} & \Conorm_{K/R}(D'' + (\deg D') \frakq) \\
      {}={} & \Conorm_{K/R}(D'') + (\deg D') \Conorm_{K/R}(\frakq) \\
      {}={} & (x)_K + (\deg D') \Conorm_{K/R}(\frakq)
    \end{align*}
    Note that $\nu_\frakp(\Conorm_{K/R}(D')) = 0$, whence $\Conorm_{K/R}(D') \le D$ and $D' \ge 0$
    imply $\nu_\frakq(D') = 0$. Therefore, $\nu_\frakq(D'') = -\deg D' \le 0$, whence $(x)_K \le
    D$. But this implies $x^{-1} \in L(D) = k$, whence $(x) = 0$, forcing $D' = 0$.
  \end{proof}
  
  Even though we restrict to hyperelliptic function fields, the previous results hold as well for
  general function fields. But from now on, we need that $K$ is hyperelliptic. Let us first handle
  the case~$g = 1$.
  
  \begin{proposition}
    If $g = 1$, let $E$ denote the set of places of degree~one. Then \[ E \to \Red_\frakp(K), \quad
    \frakq \mapsto \begin{cases} 0 & \text{if } \frakq = \frakp, \\ \frakq & \text{if } \frakq \neq
      \frakp \end{cases} \] is a bijection. \qed
  \end{proposition}
  
  It turns out that in case~$g > 1$, the converse of the lemma holds as well. Let $R$ be the unique
  rational subfield of index~$2$.
  
%
  \begin{theorem}
    \cite[p.~306, Section~14.1.2]{hehcc}
    \label{reducedtheorem}
    Assume that $g > 1$. Let $D \in \Div(K)$ be a divisor with $D \ge 0$, $\deg D \le g$,
    $\nu_\frakp(D) = 0$ such that, if $D' \in \Div(R)$ satisfies $0 \le \Conorm_{K/R}(D') \le D$,
    then $D' \le 0$. Then $D \in \Red_\frakp(K)$. \qed
  \end{theorem}

  Therefore, we have the explicit description of $\Red_\frakp(K)$ as the set \[ \left\{ D \in
  \Div(K) \;\middle|\; \begin{matrix} D \ge 0, \quad \nu_\frakp(D) = 0, \quad \deg D \le g, \\
    \forall D' \in \Div(R) : \Conorm_{K/R}(D') \le D \Rightarrow D' \le 0 \end{matrix} \right\}, \]
  together with a bijection~$\Red_\frakp(K) \to \Pic^0(K)$, $D \mapsto D - (\deg D) \frakp +
  \Princ(K)$. We will show in the next section how this leads to a combinatorial approach to
  describe the hole elements in $\Pic^0(K)$.
  
  \section{Relating Certain Reduced Divisors to the Infrastructure}
  \label{relatinginfrastructure}
  
  Assume that $K$ is hyperelliptic of genus~$g$ with exact constant field~$k$. Let $R$ be a rational
  subfield~$R$ of index~$[K : R] = 2$; in case $g > 1$, $R$ is unique. Let $\frakp$ be a place of
  $K$ of degree~one. We have seen that we have a bijection $\Red_\frakp(K) \to \Pic^0(K)$ and an
  explicit description of $\Red_\frakp(K)$. Let $S$ be a set of places of $K$ containing~$\frakp$;
  for convenience, let $S_1 := \{ \frakp \in S \mid \deg \frakp = 1 \}$. We are interested in the
  set \[ \Red_S(K) := \{ D \in \Red_\frakp(K) \mid \nu_{\frakp'}(D) = 0 \text{ for all } \frakp' \in
  S \}; \] more precisely, we are interested how its size compares to $\Red_\frakp(K) = \Red_{\{
  \frakp \}}(K)$.
  
  The reason why we are interested in this set is that it appears in studying infrastructures. Let
  $x \in K^*$ be an element whose poles are precisely the elements in $S$. Let $\calO_S$ be the
  integral closure of $k[x]$ in $K$; then the ``infinite places'' of the extension $K / k(x)$,
  i.e. the places of $K$ lying over the infinite place of $k(x)$, are exactly $S$. We have a
  surjection from $\Div(K)$ onto the (nonzero fractional) ideal group $\Id(\calO_S)$ of $\calO_S$,
  given by \[ \ideal_S : \sum_\frakp n_\frakp \frakp \mapsto \prod_{\frakp \not\in S'}
  (\frakm_\frakp \cap \calO_S)^{-n_\frakp}, \] where $\frakm_\frakp$ is the maximal ideal in the
  valuation ring~$\calO_\frakp$ of $\frakp$. Note that $\ideal_S(\Princ(K)) = \PId(\calO_S)$,
  i.e. the principal divisors map onto the group of non-zero fractional principal ideals of
  $\calO_S$. Now consider the map \[ \Psi : K^* \to \Z^{S \setminus \{ \frakp \}}, \qquad f \mapsto
  (-\nu_\frakp(f))_{\frakp \in S}. \] Fix a divisor~$D$ and the corresponding ideal~$\fraka =
  \ideal_S(D)$; we consider the set of reduced divisors in $\Red_\frakp(K)$ resp. $\Red_S(K)$
  mapping onto ideals equivalent to $\fraka$, i.e. we consider
  \begin{align*}
    \Red_\frakp(K, D) :={} & \{ D' \in \Red_\frakp(K) \mid \ideal_S(D') \PId(\calO_S) = \ideal_S(D)
    \PId(\calO_S) \} \\
    \text{and} \; \Red_S(K, D) :={} & \{ D' \in \Red_S(K) \mid \ideal_S(D') \PId(\calO_S) =
    \ideal_S(D) \PId(\calO_S) \}.
  \end{align*}
  Define the map \[ gen_D : \Red_\frakp(K, D) \to K^* / \calO_S^*, \quad D' \to \mu \calO_S^* \text{
  if } (\tfrac{1}{\mu}) = \ideal_S(D') \ideal_S(D)^{-1}; \] then the combination
  \begin{align*}
    \Phi : \Red_\frakp(K, D) \to{} & \Z^{S \setminus \{ \frakp \}} / \Psi(\calO_S^*), \\
    D' \mapsto{} & \Psi(gen_D(D')) + (\nu_\frakp(D'))_{\frakp \in S \setminus \{ \frakp \}} +
    \Psi(\calO_S^*)
  \end{align*}
  turns out to be a bijection; the subset $\Red_S(K, D)$ of $\Red_\frakp(K, D)$ maps onto a subset
  of the torus $\Z^{S \setminus \{ \frakp \}} / \Psi(\calO_S^*)$. Now the elements of $\Red_S(K, D)$
  correspond to the reduced ideals in the ideal class of $\fraka = \ideal_S(D)$, and the image of
  $\Red_S(K, D)$ under $\Phi$ equals the image of the infrastructure in the ideal class of $\fraka$
  under the distance map. Hence, we see that the \emph{hole elements} in the infrastructure --
  namely, elements of $\Z^S / \Psi(\calO_S^*)$ which lie not in the image of $\Phi$ -- are the
  elements in $\Red_\frakp(K, D) \setminus \Red_S(K, D)$; looking at all ideal classes at once, they
  correspond to the elements in $\Red_\frakp(K) \setminus \Red_S(K)$.
  
  For that reason, the comparison of $\Red_\frakp(K)$ and $\Red_S(K)$ is related to the problem of
  counting the number of holes in the infrastructure.
  
  In case $g = 1$, i.e. $K$ is elliptic, the question can be answered easily:
  
  \begin{proposition}
    \label{genus1prop}
    Assume that $g = 1$. Then \[ Red_\frakp(K) \setminus \Red_S(K) = S_1 \setminus \{ \frakp \}. \]
  \end{proposition}
  
  \begin{proof}
    The elements of $\Red_\frakp(K)$ are $D = 0$ and $D = \frakq$, where $\frakq$ ranges over all
    rational places of $K$ except $\frakp$.
  \end{proof}
  
  \begin{corollary}
    Let $g = 1$ and $k$ be a finite field of $q$ elements. Then \[
    \frac{\abs{\Red_S(K)}}{\abs{\Red_\frakp(K)}} = 1 - \frac{\abs{S_1} - 1}{q} + \calO(q^{-3/2}) \]
    for $q \to \infty$. Hence, the probability that a random reduced divisor $D \in \Red_\frakp(K)$
    is not in $\Red_S(K)$ is approximately $\frac{\abs{S_1} - 1}{q}$.
  \end{corollary}
  
  \begin{proof}
    By the proposition, \[ \frac{\abs{\Red_S(K)}}{\abs{\Red_\frakp(K)}} = \frac{\abs{\Red_\frakp(K)}
    - (\abs{S_1} - 1)}{\abs{\Red_\frakp(K)}} = 1 - \frac{\abs{S_1} - 1}{\abs{\Pic^0(K)}}. \] Now
    Hasse-Weil gives $\abs{\Pic^0(K)} = q + 1 - t$ with $\abs{t} \le 2 \sqrt{q}$, whence, for $q \ge
    7$, \[ \abs{\frac{1}{q - t + 1} - \frac{1}{q}} \le q^{-3/2} \frac{2 + q^{-1/2}}{1 - 2 q^{-1/2} -
    q^{-1}}, \] which implies the claim.
  \end{proof}
  
  In case $g > 1$, the problem is harder. We begin with another classification of reduced
  divisors. Note that if $\frakq$ is a place of $R$ which is inert in $K$, and if $\frakp$ is a
  place of $K$ lying above $\frakq$, then $\deg \frakp = 2 \deg \frakq \ge 2$; in particular, no
  place of degree~one of $K$ lies above a place of $R$ inert in $K$. Again, let $\sigma$ denote the
  unique non-trivial $R$-automorphism of $K$. One directly obtains the following classification:
  
  \begin{lemma}
    Let $\frakp$ be a place of $K$ of degree~one.
    \begin{enuma}
      \item If $\sigma(\frakp) = \frakp$, then $\frakp \cap R$ ramifies in $K$ and
      $\Conorm_{R/K}(\frakp \cap R) = 2 \frakp$.
      \item If $\sigma(\frakp) \neq \frakp$, then $\frakp \cap R$ splits in $K$ and
      $\Conorm_{R/K}(\frakp \cap R) = \frakp + \sigma(\frakp)$.
    \end{enuma}
    In case $\deg \frakp > 1$, one can also have that $\frakp \cap R$ is inert in $K$ in case
    $\sigma(\frakp) = \frakp$. \qed
  \end{lemma}
  
  We can now reformulate our explicit description of the reduced divisors using signatures:
  
  \begin{proposition}
    \label{combinatoricaldescriptionofreds}
    Let $D \in \Div(K)$. Then $D \in \Red_S(K)$ if, and only if,
    \begin{enumi}
      \item $\deg D \le g$;
      \item $\nu_\frakq(D) = 0$ for all $\frakq \in S$ and
      \item for all places $\frakq$ of $R$, \[ \{ \nu_\frakp(D) \mid \frakp \cap R = \frakq \}
      \in \begin{cases} \{ \{ 0 \} \} & \text{if } \frakq \text{ is inert in } K, \\ \{ \{ 0 \}, \{
        1 \} \} & \text{if } \frakq \text{ ramifies in } K, \\ \{ \{ 0, 0 \}, \{ 0, 1 \}, \dots, \{
        0, g \} \} & \text{if } \frakq \text{ splits in } K. \end{cases} \]
    \end{enumi}
  \end{proposition}
  
  \begin{proof}
    Clearly, the condition that $D' \in \Div(R)$ with $\Conorm_{K/R}(D') \le D$ implies $D' = 0$ is
    local, i.e. it suffices to check it for $D' = \frakq$ for all places~$\frakq$ of $R$.
    
    Let $\frakq$ be a place of $R$ and let $\frakp_1, \frakp_2$ be all places of $K$ lying above
    $\frakq$ (with $\frakp_1 = \frakp_2$ possible). If $\signature(\frakq) = (1, 2)$, then
    $\Conorm_{K/R}(\frakq) = \frakp_1 = \frakp_2$. If $\signature(\frakq) = (2, 1)$, then
    $\Conorm_{K/R}(\frakq) = 2 \frakp_1 = 2 \frakp_2$. If $\signature(\frakq) = (1, 1, 1, 1)$, then
    $\Conorm_{K/R}(\frakq) = \frakp_1 + \frakp_2$. This, together with $\deg D \le g$, shows that
    the above listed possibilities for $\{ \nu_{\frakp_1}(D), \nu_{\frakp_2}(D) \}$ correspond to
    the cases where $\Conorm_{K/R}(\frakq) \not\le D$.
  \end{proof}
  
  We have seen how the hole elements of an infrastructure correspond to the set $\Red_S(K) \setminus
  \Red_\frakp(K)$, and we obtained an explicit combinatorial description of $\Red_S(K)$. This will
  be used in the next sections to obtain information on $\abs{\Red_S(K) \setminus
  \Red_\frakp(K)}$. Finally, we have investigated the case of $g = 1$ and shown that in this case
  our main result is true.
  
  \section{Generating Functions for $\abs{\Red_S(K)}$}
  \label{genfuns}
  
  Let $S$ be an arbitrary set of places of $K$. We want to consider the subset of divisors of
  $\Red_S(K)$ of a fixed degree, and describe their quantity using a generating function. Since this
  does not make that much sense for finite sequences, we extend the definition of $\Red_S(K)$ to
  include divisors of higher degree; these additional elements are not relevant for computational
  reasons, but allow to relate the so obtained generating functions of $\Red_S(K)$ with the zeta
  function of $K$.
  
  We begin with defining a filtration $\Red_S(K) = \bigcup_{d=0}^g \Red_S^d(K)$, where $\Red_S^d(K)$
  contains the reduced divisors in $\Red_S(K)$ of degree~$d$, and extending $\Red_S^d(K)$ for $d >
  g$. For that, define \[ \Red_S^d(K) := \left\{ D \in \Div(K) \;\middle| \begin{matrix} D \ge 0, \;
    \deg D = d, \; \forall \frakq \in S : \nu_\frakp(D) = 0 \hfill \\ \forall D' \in \Div(R) :
    (\Conorm_{K/R}(D') \le D \Rightarrow D' \le 0)) \end{matrix} \right\} \] for any $d \in
  \N$. Then, the classification of Proposition~\ref{combinatoricaldescriptionofreds} also holds: 
  
  \begin{proposition}
    \label{combinatoricaldescriptionofreds2}
    Let $D \in \Div(K)$. Then $D \in \Red_S^d(K)$ if, and only if,
    \begin{enumi}
      \item $\deg D = d$;
      \item $\nu_\frakq(D) = 0$ for all $\frakq \in S$ and
      \item for all places $\frakq$ of $R$, \[ \{ \nu_\frakp(D) \mid \frakp \cap R = \frakq \}
      \in \begin{cases} \{ \{ 0 \} \} & \text{if } \frakq \text{ is inert in } K, \\ \{ \{ 0 \}, \{
        1 \} \} & \text{if } \frakq \text{ ramifies in } K, \\ \{ \{ 0, 0 \}, \{ 0, 1 \}, \dots, \{
        0, g \} \} & \text{if } \frakq \text{ splits in } K. \end{cases} \] \qed
    \end{enumi}
  \end{proposition}

  Moreover, in the case that $\frakp \in S$, we have the disjoint union \[ \Red_S(K) =
  \bigcup_{d=0}^g \Red_S^d(K). \] Consider $C_n(S) := \abs{\Red_S^n(K)}$; we are interested in the
  generating function \[ h_S(t) := \sum_{n=0}^\infty C_n(S) t^n \in \Q[[t]] \] and its relation to
  $h_\emptyset(t)$. We begin with a statement on the relation of $\Red_S^n(K)$ if we modify $S$ in
  certain ways.
  
  \begin{proposition}
    \label{SSprimerelation}
    Let $S' = S \cup \{ \frakp_1, \frakp_2 \}$ with $\frakp_i \not\in S$, and let $d \in \N$.
    \begin{enuma}
      \item Assume that $\frakp_1 \neq \frakp_2 = \sigma(\frakp_1)$. Then we have the disjoint
      union \[ \Red_S^d(K) = \bigcup_{i=0}^\infty \{ D + i \frakp_1, D + i \frakp_2 \mid D \in
      \Red_{S'}^{d - i \deg \frakp_1}(K) \}. \]
      \item Assume that $\frakp_1 = \frakp_2 = \sigma(\frakp_1)$, and that $\frakp_1 \cap R$ is not
      inert in $K$. Then we have the disjoint union \[ \Red_S^d(K) = \bigcup_{i=0}^1 \{ D + i
      \frakp_1 \mid D \in \Red_{S'}^{d - i \deg \frakp_1}(K) \}. \]
      \item Assume that $\frakp_1 = \frakp_2 = \sigma(\frakp_1)$, and that $\frakp_1 \cap R$ is
      inert in $K$. Then we have \[ \Red_S^d(K) = \Red_{S'}^d(K). \]
      \item Assume that $\frakp_1 = \frakp_2$ and $\sigma(\frakp_1) \in S$. Then we have the
      disjoint union \[ \Red_S^d(K) = \bigcup_{i=0}^\infty \{ D + i \frakp_1 \mid D \in \Red_{S'}^{d
      - i \deg \frakp_1}(K) \}. \]
    \end{enuma}
  \end{proposition}
  
  \begin{proof}
    This is clear from the generalization of Proposition~\ref{combinatoricaldescriptionofreds}.
  \end{proof}
  
  This allows us to state how to obtain $h_S(t)$ from $h_{S'}(t)$ in these cases:
  
  \begin{theorem}
    Let $S' = S \cup \{ \frakp_1, \frakp_2 \}$ with $\frakp_i \not\in S$.
    \begin{enuma}
      \item Assume that $\frakp_1 \neq \frakp_2 = \sigma(\frakp_1)$. Then \[ h_{S'}(t) = \frac{1 -
      t^{\deg \frakp_1}}{1 + t^{\deg \frakp_1}} h_S(t). \]
      \item Assume that $\frakp_1 = \frakp_2 = \sigma(\frakp_1)$, and that $\frakp_1 \cap R$ is not
      inert in $K$. Then \[ h_{S'}(t) = \frac{1}{1 + t^{\deg \frakp_1}} h_S(t). \]
      \item Assume that $\frakp_1 = \frakp_2 = \sigma(\frakp_1)$, and that $\frakp_1 \cap R$ is
      inert in $K$. Then \[ h_{S'}(t) = h_S(t). \]
      \item Assume that $\frakp_1 = \frakp_2$ and $\sigma(\frakp_1) \in S$. Then \[ h_{S'}(t) = (1 -
      t^{\deg \frakp_1}) h_S(t). \]
    \end{enuma}
  \end{theorem}
  
  \begin{proof}
    Using the proposition, we obtain
    \begin{align*}
      \text{(a)} \qquad & C_d(S) = C_d(S') + 2 \sum_{i=1}^\infty C_{d - i \deg \frakp_1}(S'), \\
      \text{(b)} \qquad & C_d(S) = C_d(S') + C_{d - \deg \frakp_1}(S'), \\
      \text{(c)} \qquad & C_d(S) = C_d(S') \qquad \text{and} \\
      \text{(d)} \qquad & C_d(S) = \sum_{i=0}^\infty C_{d - i \deg \frakp_1}(S').
    \end{align*}
    Therefore,
    \begin{align*}
      \text{(a)} \qquad & h_S(t) = h_{S'}(t) \left( 1 + 2 \sum_{i=1}^\infty t^{i \deg
      \frakp_1} \right), \\
      \text{(b)} \qquad & h_S(t) = h_{S'}(t) (1 + t^{\deg \frakp_i}) \\
      \text{(c)} \qquad & h_S(t) = h_{S'}(t) \qquad \text{and} \\
      \text{(d)} \qquad & h_S(t) = h_{S'}(t) \sum_{i=0}^\infty t^{i \deg \frakp_1}.
    \end{align*}
    Now, the geometric series gives \[ \sum_{i=0}^\infty t^{i \deg \frakp_1} = \frac{1}{1 - t^{\deg
    \frakp_1}}, \] and multiplying with $1 + t^{\deg \frakp_1}$ gives \[ 1 + 2 \sum_{i=1}^\infty
    t^{i \deg \frakp_1} = \frac{1 + t^{\deg \frakp_1}}{1 - t^{\deg \frakp_1}}. \] Plugging this in
    and solving for $h_{S'}(t)$, we obtain the claim.
  \end{proof}
  
  With this, we can explicitly describe $h_S(t)$ in terms of $h_\emptyset(t)$ when $S$ is a finite
  set not containing places lying above inert places of $R$.
  
  \begin{corollary}
    \label{hSexplicitdescription}
    Let $S$ be a finite set of places of $K$, containing no places lying above inert places of
    $R$. For $i \in \N$, let $S_i = \{ \frakp \in S \mid \deg \frakp = i \}$, and let
    \begin{itemize}
      \item $n_i = \abs{\{ \frakp \in S_i \mid \sigma(\frakp) = \frakp \}}$,
      \item $\ell_i = \frac{1}{2} \abs{\{ \frakp \in S_i \mid \frakp \neq \sigma(\frakp) \in S_i
      \}}$,
      \item $m_i = \abs{\{ \frakp \in S_i \mid \sigma(\frakp) \not\in S \}}$.
    \end{itemize}
    Then \[ h_S(t) = h_\emptyset(t) \cdot \prod_{i=1}^\infty \left( (1 - t^i)^{\ell_i} (1 +
    t^i)^{-\ell_i - n_i - m_i} \right). \]
  \end{corollary}
  
  \begin{proof}
    Let $S' = S \cup \sigma(S)$. If $S'_i$ and $n'_i, \ell'_i, m'_i$ are defined in a similar
    manner, than $n'_i = n_i$, $\ell'_i = \ell_i + m_i$, $m'_i = 0$. Using the Theorem, we obtain
    \begin{align*}
      h_{S'}(t) ={} & h_\emptyset(t) \cdot \prod_{i=1}^\infty \left( \frac{1 - t^i}{1 + t^i}
      \right)^{\ell'_i} \cdot \prod_{i=1}^\infty \left( \frac{1}{1 + t^i} \right)^{n'_i} \\
      {}={} & h_\emptyset(t) \cdot \prod_{i=1}^\infty \left( (1 - t^i)^{\ell'_i} (1 + t^i)^{-\ell'_i
      - n'_i} \right) \\
      {}={} & h_\emptyset(t) \cdot \prod_{i=1}^\infty \left( (1 - t^i)^{\ell_i + m_i} (1 +
      t^i)^{-\ell_i - m_i - n_i} \right).
    \end{align*}
    Using the Theorem a second time and the fact that \[ S' \setminus S = \{ \sigma(\frakp) \in S
    \mid \sigma(\frakp) \not\in S \}, \] we get \[ h_{S'}(t) = h_S(t) \cdot \prod_{i=1}^\infty (1 -
    t^i)^{m_i}. \] Putting everything together, we have \[ h_S(t) = h_{S'}(t) \prod_{i=1}^\infty (1
    - t_i)^{-m_i} = h_\emptyset(t) \cdot \prod_{i=1}^\infty \left( (1 - t^i)^{\ell_i} (1 +
    t^i)^{-\ell_i - m_i - n_i} \right), \] what we had to show.
  \end{proof}
  
  Next, we want to find bounds for the coefficients of the Taylor expansion of the rational
  functions involved in describing the relation of $h_S(t)$ to $h_\emptyset(t)$. For that, we need a
  small lemma on formal power series.
  
  \begin{lemma}
    Let $f = \sum_{n=0}^\infty a_n t^n, g = \sum_{n=0}^\infty b_n t^n \in \C[[t]]$ and $a, b \in
    \R_{\ge 0}$ such that $\abs{a_n} \le a^n$, $\abs{b_n} \le b^n$. If $f g = \sum_{n=0}^\infty
    c_n$, then $\abs{c_n} \le (a + b)^n$ for all $n \in \N$. In case $a_0 = b_0 = 1$, we have $c_0 = 1$ and $c_1 = a_1 + b_1$.
  \end{lemma}
  
  \begin{proof}
    We have $c_n = \sum_{i=0}^n a_i b_{n-i}$, whence \[ \abs{c_n} \le \sum_{i=0}^n \abs{a_i}
    \abs{b_{n-i}} \le \sum_{i=0}^n \binom{n}{i} a^i b^{n-i} = (a + b)^n. \] Finally, $a_0 = b_0 = 1$
    clearly implies $c_0 = 1$ and $c_1 = a_1 + b_1$.
  \end{proof}
  
  
  
  
  
  Using this, we obtain how two generating functions $h_S(t)$ and $h_{S'}(t)$ differ in case $S'
  \subseteq S$:
  
  \begin{corollary}
    \label{coefficientboundcorollary}
    Assume that $S$ is finite and contains no places lying above inert places of $R$, and let $S'
    \subseteq S$ be a subset. Let \[ h_S(t) = h_{S'}(t) \sum_{n=0}^\infty a_n t^n. \] Then $a_0 =
    1$, $a_1 = -(\abs{S_1} - \abs{S'_1})$ and $\abs{a_n} \le (\abs{S} - \abs{S'})^n$ for $n \in \N$.
  \end{corollary}
  
  \begin{proof}
    Define $S_i$, $S'_i$, $n_i, n'_i$, $m_i, m'_i$ and $\ell_i, \ell'_i$ as before. Then
    \[ h_S(t) = h_{S'}(t) \cdot \prod_{i=1}^\infty \left( \left(\frac{1 - t^i}{1 +
    t^i}\right)^{\ell_i - \ell'_i} \left(\frac{1}{1 + t^i}\right)^{n_i + m_i - n_i' -
    m_i'}\right). \] Note that $\ell'_i \le \ell_i$ and $n'_i \le n_i$, whence $\ell_i - \ell'_i \ge
    0$ an $n'_i - n_i \ge 0$. Moreover, $\widetilde{M}_i := \max\{ m_i - m'_i, 0 \} \le \ell_i -
    \ell'_i$ and $m'_i - m_i = \widehat{M}_i - \widetilde{M}_i$ with $\widehat{M}_i := \max\{ m'_i -
    m_i, 0 \}$. Write \[ \sum_{n=0}^\infty c_n t^n = \left(\frac{1 - t^i}{1 + t^i}\right)^{\ell_i -
    \ell'_i - \widetilde{M}_i} \left(\frac{1}{1 + t^i}\right)^{(n_i - n'_i) + \widehat{M}_i} \left(
    1 - t^i \right)^{\widetilde{M}_i}. \] Hence, by applying the lemma repeatedly, we obtain
    \begin{align*}
      c_1 ={} & \begin{cases} -(\abs{S_1} - \abs{S'_1}) & \text{if } i = 1, \\ 0 &
        \text{otherwise} \end{cases} \\
      \text{and} \quad \abs{c_n} \le{} & [2 (\ell_i - \ell'_i - \widetilde{M}_i) + (n_i - n'_i) +
      \widehat{M}_i + \widetilde{M}_i]^n = (\abs{S_i} - \abs{S'_i})^n.
    \end{align*}
  \end{proof}
  
  We can now make an explicit statement on the number of hole elements $\abs{\Red_\frakp(K)
  \setminus \Red_S(K)}$ for all hyperelliptic curves.
  
  \begin{corollary}
    \label{reddifferenceformula}
    Assume that $g \ge 1$, let $S$ be finite and let $\frakp \in S$ be a place of degree~one. There
    exist efficiently computable $c_1, \dots, c_g \in \Z$ such that \[ \abs{\Red_\frakp(K) \setminus
    \Red_S(K)} = (\abs{S_1} - 1) C_{g-1}(\{ \frakp \}) - \sum_{j=0}^{g-2} \sum_{i=1}^{g-j} c_i
    C_j(\{ \frakp \}) \] and
    \begin{align*}
      & \abs{\abs{\Red_\frakp(K) \setminus \Red_S(K)} - (\abs{S_1} - 1) C_{g-1}(\{ \frakp \})} \\
      {}\le{} & \sum_{j=0}^{g-2} (g - j) (\abs{S} - 1)^{g-j} C_j(\{ \frakp \}).
    \end{align*}
  \end{corollary}
  
  \begin{proof}
    By Corollary~\ref{coefficientboundcorollary}, we have that $h_S(t) = h_{\{ \frakp \}}(t) \cdot
    \sum_{n=0}^\infty a_n t^n$ with $a_0 = 1$, $a_1 = -(\abs{S_1} - 1)$ and $\abs{a_n} \le (\abs{S}
    - 1)^n$. Therefore, \[ h_{\{ \frakp \}}(t) - h_S(t) = \left( \sum_{n=0}^\infty C_n(\{ \frakp \})
    t^n \right) \cdot \left( (\abs{S_1} - 1) - \sum_{n=2}^\infty a_n t^n \right). \] Now the
    coefficient of $t^d$ equals $\bigl|\Red_\frakp^d(K)\bigr| - \bigl|\Red_S^d(K)\bigr|$, whence
    \begin{align*}
      & \abs{\Red_\frakp(K) \setminus \Red_S(K)} = -\sum_{d=1}^g \sum_{i=0}^{d-1} C_i(\{ \frakp \})
      a_{d-i} \\
      {}={} & -\sum_{i=0}^{g-1} \sum_{d=i+1}^g a_{d-i} C_i(\{ \frakp \}) = -a_1 C_{g-1}(\{ \frakp
      \}) - \sum_{i=0}^{g-2} \sum_{d=1}^{g-i} a_d C_i(\{ \frakp \}),
    \end{align*}
    which implies the first equality. By using $\abs{a_n} \le (\abs{S} - 1)^n$, we get
    \begin{align*}
      & \abs{\abs{\Red_\frakp(K) \setminus \Red_S(K)} - (\abs{S_1} - 1) C_{g-1}(\{ \frakp \})} \\
      {}\le{} & \sum_{i=0}^{g-2} \sum_{d=1}^{g-i} (\abs{S} - 1)^d C_i(\{ \frakp \}) \le
      \sum_{i=0}^{g-2} (g - i) (\abs{S} - 1)^{g-i} C_i(\{ \frakp \}).
    \end{align*}
  \end{proof}
  
  In this section we found a description of the generating function~$h_S(t)$ of $\Red_S(K)$ in terms
  $h_\emptyset(t)$ and a rational factor, of which we have information on its coefficients in the
  Taylor expansion. This allowed us to give a bound on $\abs{\abs{\Red_\frakp(K) \setminus
  \Red_S(K)} - (\abs{S_1} - 1) C_{g-1}(\{ \frakp \})}$ in terms of $\abs{S} - 1$ and $C_i(\{ \frakp
  \})$, $0 \le i \le g$.
  
  Our next goal is to obtain information on $h_{\{ \frakp \}}(t)$, i.e. on the $C_i(\{ \frakp
  \})$'s; then, we can combine this information with the above result to obtain our main result on
  $\abs{\Red_\frakp(K) \setminus \Red_S(K)}$.
  
  \section{Counting Reduced Divisors of Certain Degrees}
  \label{countingreddivs}
  
  In this section, we want to obtain information on $h_{\{ \frakp \}}(t)$. In particular, we show
  that all $h_S(t)$ are rational as long as $S$ is finite and relate $h_{\{ \frakp \}}(t)$ to the
  $L$-polynomial of $K$.
  
  Let $k = \F_q$, the field of $q$ elements. We begin considering $C_d(\emptyset) =
  \bigl|\Red_\emptyset^d(K)\bigr|$ for all $d \in \N$. Later, we will relate the $C_d(\emptyset)$'s
  with the $C_d(\{ \frakp \})$'s.
  
  For $d \in \N$, we also consider the sets 
  \begin{align*}
    \Div_+^d(K) :={} & \{ D \in \Div(K) \mid D \ge 0, \deg D = d \} \\ \text{and} \qquad
    \Div_+^d(R) :={} & \{ D \in \Div(R) \mid D \ge 0, \deg D = d \};
  \end{align*}
  set $A_n(K) := \bigl|\Div_+^d(K)\bigr|$ and $A_n(R) := \bigl|\Div_+^d(R)\bigr|$.
  
  \begin{proposition}
    Let $d \ge 0$. For every $D \in \Div_+^d(K)$, there exists a unique integer $r$ with $0 \le r
    \le d/2$ and two unique divisors $D_R \in \Div_+^r(R)$, $D_K \in \Red_\emptyset^{d - 2 r}(K)$
    such that $D = \Conorm_{K/R}(D_R) + D_K$.
  \end{proposition}
  
  \begin{proof}
    For $D \in \Div_+^d(K)$, consider \[ A(D) := \{ D_R \in \Div_+(R) \mid D_R \ge 0, \;
    \Conorm_{K/R}(D_R) \le D \}. \] This set turns out to be a finite lattice when ordered with
    $\le$, whence it has a maximal element, say $D_R$. Then $D_K := D - \Conorm_{K/R}(D_R) \ge 0$
    and, if $r := \deg D_R$, by the maximality of $D_R$, $D_K \in \Red^{d - 2 r}(K)$. The uniqueness
    is clear from the lattice structure of $A(D)$.
  \end{proof}
  
  \begin{corollary}
    For $d \ge 0$, we have \[ A_d(K) = \sum_{r=0}^{\floor{d/2}} A_r(R) C_{d - 2 r}(\emptyset). \]
    \qed
  \end{corollary}
  
  The zeta function of $K$ is given by \[ Z_K(t) := \sum_{n=0}^\infty A_n(K) t^n, \] and the zeta
  function of $R$ is given by \[ Z_R(t) := \sum_{n=0}^\infty A_n(R) t^n = \frac{1}{(1 - t) (1 - q
  t)}. \] (See \cite[Chapter~5]{stichtenoth}.)
  
  Consider the formal power series $h_\emptyset(t) = \sum_{d=0}^\infty C_d(\emptyset) t^d \in
  \Q[[t]]$. The following result shows its relation to the zeta function of $K$ and the zeta
  function of $R$:
  
  \begin{lemma}
    Let $f(t) = Z_R(t^2) \cdot h_\emptyset(t) \in \Q[[t]]$ as a formal power series; write $f(t) =
    \sum_{n=0}^\infty a_n t^n$. Then, for $d \ge 0$, $a_n = A_n(K)$, i.e. $f(t) = Z_K(t)$.
  \end{lemma}
  
  \begin{proof}
    We have \[ f = \biggl(\sum_{n=0}^\infty A_n(R) t^{2 n}\biggr) \cdot \biggl(\sum_{m=0}^\infty
    C_m(\emptyset) t^m\biggr). \] Hence, the coefficient of $t^d$ in the product is given by
    $\sum_{2 n + m = d} A_n(R) C_m(\emptyset)$. But this means $n \le \floor{d/2}$, i.e. we can
    write this sum as $\sum_{n=0}^{\floor{d/2}} A_n(R) C_{d - 2 n}(\emptyset)$, which equals
    $A_d(K)$ by the previous corollary.
  \end{proof}
  
  Therefore, we see that \[ h_\emptyset(t) = \frac{Z_K(t)}{Z_R(t^2)} = (1 - t^2) (1 - q t^2)
  Z_K(t). \] Now $Z_K(t)$ is a rational function as well: by \cite[p.~193,
  Theorem~V.1.15]{stichtenoth}, $Z_K(R) = \frac{L_K(t)}{(1 - t) (1 - q t)}$, where $L_K \in \Z[t]$
  with $\deg L_K = 2 g$; the polynomial $L_K$ is called the \emph{$L$-polynomial} of $K$. Hence, \[
  h_\emptyset(t) = \frac{(1 - t^2) (1 - q t^2) L_K(t)}{(1 - t) (1 - q t)} = \frac{(1 + t) (1 - q
  t^2) L_K(t)}{1 - q t} \] is a rational function with a simple pole in $t = q^{-1}$; in particular,
  $h_\emptyset(t)$ is a convergent power series with radius of convergence $q^{-1}$. Therefore, we
  obtain:
  
  \begin{theorem}
    For any finite set $S$ of places of $K$, $h_S(t)$ is a convergent power series with radius of
    convergence~$q^{-1}$. In particular, \[ h_\emptyset(t) = \frac{(1 + t) (1 - q t^2) L_K(t)}{1 - q
    t}, \] where $L_K$ is the $L$-polynomial of $K$. \qed
  \end{theorem}
  
  Important as well is the fact that $(1 - q t) h_\emptyset(t)$ is a polynomial of degree~$2 g + 3$,
  namely $(1 + t) (1 - q t^2) L_K(t)$. Write \[ L_K(t) = \sum_{i=0}^{2 g} a_i t^i. \] Then $a_0 =
  1$, $a_{2 g} = q^g$ and $a_{2 g - i} = q^{g - i} a_i$ for $0 \le i \le g$. Moreover, $a_1 = N - (q
  + 1)$, where $N = A_1(K) = \abs{\Div_+^1(K)}$ is the number of rational places of degree~one.
  
  \begin{lemma}
    We have \[ (1 - q t) h_\emptyset(t) = 1 + (C_1(\emptyset) - q) t + \sum_{d=2}^\infty
    (C_d(\emptyset) - q C_{d-1}(\emptyset)) t^d \] and
    \begin{align*}
      (1 + t) (1 - q t^2) L_K(t) ={} & 1 + (a_1 + 1) t + (a_1 + a_2 - q) t^2 \\
      {}+{} & \sum_{i=3}^{2 g} (a_i + a_{i-1} - q a_{i-2} - q a_{i-3}) t^i \\
      {}+{} & (q^g - q a_{2 g - 1} - q a_{2 g - 2}) t^{2 g + 1} \\
      {}-{} & q (q^g + a_{2 g - 1}) t^{2 g + 2} - q^{g+1} t^{2 g + 3}.
    \end{align*}
  \end{lemma}
  
  \begin{proof}
    This is an easy and direct computation.
  \end{proof}
  
  Using $(1 - q t) h_\emptyset(t) = (1 + t) (1 - q t^2) L_K(t)$ and comparing coefficients, we obtain:
  
  \begin{corollary}
    \label{hatCexplicit}
    For $d \in \{ 4, \dots, 2 g \}$, we have
    \begin{align*}
      C_0(\emptyset) ={} & 1, \\
      C_1(\emptyset) ={} & q + a_1 + 1, \\
      C_2(\emptyset) ={} & q^2 + a_1 q (1 + q^{-1}) + a_2, \\
      C_3(\emptyset) ={} & q^3 + a_1 q^2 + a_2 q (1 + q^{-1}) + a_3 - q \qquad \text{and} \\
      C_d(\emptyset) ={} & q^d + \sum_{i=1}^{d-3} a_i q^{d - i} (1 - q^{-2}) + a_{d-2} q^2 + a_{d-1}
      q (1 + q^{-1}) + a_d - q^{d - 2}.
    \end{align*}
  \end{corollary}
  
  \begin{proof}
    The equalities for $d \le 2$ follow directly from the lemma. For $d \ge 3$, we have \[
    C_d(\emptyset) - q^d = q (C_{d-1}(\emptyset) - q^{d - 1}) + a_d + a_{d-1} - q a_{d-2} - q
    a_{d-3}. \] Plugging in $d = 3$ and the formula for $C_2(\emptyset)$, we obtain \[
    C_3(\emptyset) - q^3 = q^2 a_1 + q (1 + q^{-1}) a_2 + a_3 - q. \] Now, for $d = 4$, we similarly
    obtain \[ C_4(\emptyset) - q^4 = q^3 (1 - q^{-2}) a_1 + q^2 a_2 + q (1 + q^{-1}) a_3 + a_4 -
    q^2. \] Now let $d \ge 4$; then, using induction,
    \begin{align*}
      C_{d+1}(\emptyset) - q^{d+1} ={} & q (C_d(\emptyset) - q^d) + a_{d+1} + a_d - q a_{d-1} - q a_{d-2} \\
      {}={} & \sum_{i=1}^{d-3} q^{d+1-i} (1 - q^{-2}) a_i + q^3 a_{d-2} + q^2 (1 + q^{-1}) a_{d-1}
      \\
      {}+{} & q a_d - q^{d-1} + a_{d+1} + a_d - q a_{d-1} - q a_{d-2} \\
      {}={} & \sum_{i=1}^{d-3} q^{d+1-i} (1 - q^{-2}) a_i + q^3 (1 - q^{-2}) a_{d-2} + q^2 a_{d-1}
      \\
      {}+{} & q (1 + q^{-1}) a_d + a_{d+1} - q^{d-1},
    \end{align*}
    what we had to show.
  \end{proof}
  
  We have further information on the integers~$a_i$. The result we need in the following are the
  Hasse-Weil bounds:
  
  \begin{proposition}[Hasse-Weil Bounds]
    \label{hasseweilbounds}
    For $i = 0, \dots, 2 g$, we have $\abs{a_i} \le \binom{2 g}{i} q^{i/2}$.
  \end{proposition}
  
  \begin{proof}
    By Hasse-Weil \cite[p.~193, Theorem~V.1.15 and p.~197, Theorem~V.2.1]{stichtenoth}, $L_K(t) =
    \prod_{i=1}^{2 g} (1 - \alpha_i t)$ with $\abs{\alpha_i} = q^{1/2}$. Therefore, \[ a_i = (-1)^i
    \underset{1 \le j_1 < \dots < j_i \le 2 g}{\sum \cdots\cdots \sum} \prod_{t=1}^i
    \alpha_{j_t}. \] The sum has $\binom{2 g}{i}$ terms, whence we obtain the specified bound.
  \end{proof}

  Using them, we can make explicit statements on the cardinality of $C_d(\{ \frakp \})$ for the
  case~$\deg \frakp = 1$:
  
  \begin{theorem}
    \label{boundsonC(frakp)}
    Let $S = \{ \frakp \}$ with $\deg \frakp = 1$. Then, for $d \in \{ 1, \dots, g \}$,
    \begin{align*}
      C_d(S) \le{} & \sum_{i=0}^d \binom{2 g}{i} q^{d - i/2} \\ \text{and} \qquad 
      \abs{C_d(S) - q^d} \le{} & 2 q^{d - 1} + \sum_{i=1}^d \binom{2 g}{i} q^{d - i/2}.
    \end{align*}
  \end{theorem}
  
  For the proof, we need a rather technical lemma.
  
  \begin{lemma}
    Let $S = \{ \frakp \}$ with $\deg \frakp = 1$.
    \begin{enuma}
      \item If $\sigma(\frakp) = \frakp$,
      \begin{align*}
        C_0(S) ={} & 1, \\
        C_1(S) ={} & q + a_1 \\ \text{and} \qquad
        C_d(S) ={} & \sum_{i=0}^{d-2} a_i q^{d - i} (1 - q^{-1}) + a_{d-1} q + a_d
      \end{align*}
      for $d \in \{ 2, \dots, 2 g \}$.
      \item If $\sigma(\frakp) \neq \frakp$,
      \begin{align*}
        C_0(S) ={} & 1, \\
        C_1(S) ={} & q + a_1, \\
        C_2(S) ={} & q^2 (1 - q^{-1} + q^{-2}) + a_1 q + a_2, \\
        C_3(S) ={} & q^3 (1 - q^{-1} - q^{-2}) + a_1 q^2 (1 - q^{-1} - q^{-2}) + a_2 q + a_3 \\ \text{and} \quad
        C_d(S) ={} & \sum_{i=0}^{d-4} a_i q^{d - i} (1 - q^{-1} - q^{-2} + q^{-3}) + a_{d-3} q^3 (1 -
        q^{-1} - q^{-2}) \\
        {}+{} & a_{d-2} q^2 (1 - q^{-1} - q^{-2}) + a_{d-1} q + a_d.
      \end{align*}
      for $d \in \{ 4, \dots, 2 g \}$.
    \end{enuma}
  \end{lemma}
  
  \begin{proof}\hfill
    \begin{enuma}
      \item By Proposition~\ref{SSprimerelation}~(b), $C_d(\emptyset) = C_d(S) + C_{d-1}(S)$, whence
      $C_0(S) = C_0(\emptyset)$ and, for $d > 0$, $C_d(S) = C_d(\emptyset) - C_{d-1}(S)$. Hence,
      with Corollary~\ref{hatCexplicit}, one obtains
      \begin{align*}
        C_0(S) ={} & 1, \\
        C_1(S) ={} & q + a_1, \\
        C_2(S) ={} & q^2 - q + a_1 q + a_2 \\ \text{and} \qquad
        C_3(S) ={} & q^3 - q^2 + a_1 q^2 (1 - q^{-1}) + a_2 q + a_3.
      \end{align*}
      For $d > 3$, using induction and Corollary~\ref{hatCexplicit}, we see that
      \begin{align*}
        & C_d(S) = C_d(\emptyset) - C_{d-1}(S) \\
        {}={} & \left[ q^d + \sum_{i=1}^{d-3} a_i q^{d - i} (1 - q^{-2}) + a_{d-2} q^2 + a_{d-1} q
        (1 + q^{-1}) + a_d - q^{d - 2} \right] \\
        {}-{} & \left[ \sum_{i=0}^{d-3} a_i q^{d - 1 - i} (1 - q^{-1}) + a_{d-2} q + a_{d-1} \right]
        \\
        {}={} & \sum_{i=0}^{d-2} a_i q^{d - i} (1 - q^{-1}) + a_{d-1} q + a_d,
      \end{align*}
      what we had to show.
      \item By Proposition~\ref{SSprimerelation}~(a), $C_d(\emptyset) = \sum_{i=0}^d C_i(S)$, whence
      $C_0(S) = C_0(\emptyset)$ and $C_d(S) = C_d(\emptyset) - C_{d-1}(\emptyset)$ for $d \in \{ 1,
      \dots, g \}$. Hence, with Corollary~\ref{hatCexplicit},
      \begin{align*}
        C_0(S) ={} & 1, \\
        C_1(S) ={} & q + a_1, \\
        C_2(S) ={} & q^2 (1 - q^{-1} + q^{-2}) + a_1 q + a_2, \\
        C_3(S) ={} & q^3 (1 - q^{-1} - q^{-2}) + a_1 q^2 (1 - q^{-1} - q^{-2}) + a_2 q + a_3, \\
        C_4(S) ={} & q^4 (1 - q^{-1} + q^{-3}) + a_1 q^3 (1 - q^{-1} + q^{-3}) \\
        {}+{} & a_2 q^2 (1 - q^{-1}) + a_3 q + a_4.
      \end{align*}
      For $d > 4$,
      \begin{align*}
        C_d(S) ={} & C_d(\emptyset) - C_{d-1}(\emptyset) \\
        {}={} & \sum_{i=0}^{d-4} a_i q^{d - i} (1 - q^{-1} - q^{-2} + q^{-3}) + a_{d-3} q^3 (1 -
        q^{-1} - q^{-2}) \\
        {}+{} & a_{d-2} q^2 (1 - q^{-1} - q^{-2}) + a_{d-1} q + a_d.
      \end{align*}
      \qedhere
    \end{enuma}
  \end{proof}
  
  \begin{proof}[Proof of Theorem~\ref{boundsonC(frakp)}.]
    Note that both for $\sigma(\frakp) = \frakp$ and $\sigma(\frakp) \neq \frakp$, one quickly
    obtains from the Lemma that
    \begin{align*}
      \abs{C_d(S)} \le{} & \sum_{i=0}^d q^{d - i} \abs{a_i} \\ \text{and} \qquad 
      \abs{C_d(S) - q^d} \le{} & 2 q^{d-1} + \sum_{i=1}^d q^{d - i} \abs{a_i}
    \end{align*}
    using $0 < 1 - q^{-1} < 1$, $0 < 1 - q^{-1} + q^{-2} < 1$, $0 < 1 - q^{-1} - q^{-2} < 1$, $0 < 1
    - q^{-1} - q^{-2} + q^{-3} < 1$, $0 < 1 + q^{-1} < 2$ and $0 < 1 + q^{-1} - q^{-2} < 2$. Now
    $q^{d - i} \abs{a_i} \le \binom{2 g}{i} q^{d - i/2}$ by the Hasse-Weil bounds, whence we can
    conclude.
  \end{proof}
  
  In this section, we have shown that $h_S(t)$ is rational for every finite set of
  places~$S$. Moreover, we have given bounds on the coefficients of $h_{\{ \frakp \}}(t)$ for some
  place~$\frakp$ of degree~one. In the next section, we will combine these bounds with
  Corollary~\ref{reddifferenceformula} to obtain our first two main results.
  
  \section{Counting the Number of Hole Elements}
  \label{holecountingsect}
  
  The first main result gives an explicit bound on how much the number of hole elements deviates
  from $(\abs{S_1} - 1) q^{g - 1}$. The dominant part of the error term turns out to be $2 g
  (\abs{S} - 1) q^{g - 3/2}$, which reminds of the Hasse-Weil bound on the divisor class group:
  namely, we have $\abs{\Pic^0(K)} \in [(\sqrt{q} - 1)^{2 g}, (\sqrt{q} + 1)^{2 g}]$ and $(\sqrt{q}
  \pm 1)^{2 g} = q^g \pm 2 g q^{g - 1/2} + \dots$.
  
  \begin{theorem}
    \label{mainresult1}
    Let $S$ be a finite set of places of $K$, containing a place of degree~one and no places lying
    over inert places of $R$. Assume that $g \ge 1$ and $q^{1/2} > \abs{S} + g$. We have
    \begin{align*}
      & \abs{\abs{\Red_\frakp(K) \setminus \Red_S(K)} - (\abs{S_1} - 1) q^{g-1}} \\
      {}\le{} & 2 g (\abs{S_1} - 1) q^{g - 3/2} + 2^{2 g} g^{g - 1} (\abs{S} - 1)^2 q^{g - 2}.
    \end{align*}
  \end{theorem}
  
  For the proof, we need two technical lemmata. We also use the abbreviation $C_d := C_d(\{ \frakp
  \})$.
  
  \begin{lemma}
    Assume that $g \ge 2$ and that $q^{1/2} \ge \max\{ \abs{S} - 1, 2 \}$. We then have
    \begin{align*}
      & \abs{\abs{\Red_\frakp(K) \setminus \Red_S(K)} - (\abs{S_1} - 1) C_{g-1}} \le 2^{2 g - 2}
      g^{g - 1} (\abs{S} - 1)^2 q^{g - 2}.
    \end{align*}
  \end{lemma}
  
  \begin{proof}
    For $\abs{S} = 1$ there is nothing to show; hence, assume that $\abs{S} > 1$. Note that $\abs{S}
    - 1 \le q^{1/2} \le \tfrac{1}{2} q$ (as $q \ge 4$) gives $q - (\abs{S} - 1) \ge \frac{1}{2} q$
    and $q^{1/2} + (\abs{S} - 1) \le 2 q^{1/2}$. We show the result in three steps.
    \begin{enumi}
      \item Clearly, for $0 \le i \le j \le g - 2$, \[ \binom{2 g}{i} = \binom{j}{i} \frac{(2 g)! (j
      - i)!}{(2 g - i)! j!} \le \binom{j}{i} \frac{(2 g)! j!}{(2 g - j)! j!} \le \binom{j}{i}
      \frac{(2 g)!}{(g + 2)!}. \] We have
      \begin{align*}
        \sum_{i=0}^j \binom{2 g}{i} q^{j - i/2} \le{} & \frac{(2 g)!}{(g + 2)!} q^j \sum_{i=0}^j
        \binom{j}{i} q^{-i/2} = \frac{(2 g)!}{(g + 2)!} (q + q^{1/2})^j.
      \end{align*}
      \item We have 
      \begin{align*}
        & \sum_{j=0}^{g-2} (g - j) (\abs{S} - 1)^{g - j} \sum_{i=0}^j \binom{2 g}{i} q^{j-i/2} \\
        \overset{(i)}{{}\le{}} & g \frac{(2 g)!}{(g + 2)!} (\abs{S} - 1)^g \sum_{j=0}^{g-2}
        \left(\frac{q + q^{1/2}}{\abs{S} - 1}\right)^j \\
        {}={} & g \frac{(2 g)!}{(g + 2)!} (\abs{S} - 1)^g \frac{\left(\frac{q + q^{1/2}}{\abs{S} -
        1}\right)^{g-1} - 1}{\frac{q + q^{1/2}}{\abs{S} - 1} - 1} \\
        {}\le{} & g (2 g)^{g - 2} (\abs{S} - 1)^2 \frac{q^{g-1} (1 + q^{-1/2})^{g-1}}{q +
        q^{1/2} + 1 - \abs{S}} \\
        {}\le{} & 2^{g - 2} g^{g - 1} (\abs{S} - 1)^2 \frac{q^{g-1} 2^{g-1}}{q} = 2^{2 g - 3} g^{g -
        1} (\abs{S} - 1)^2 q^{g-2}.
      \end{align*}
      \item By Theorem~\ref{boundsonC(frakp)}, we have
      \begin{align*}
        & \abs{\abs{\Red_\frakp(K) \setminus \Red_S(K)} - (\abs{S_1} - 1) C_{g-1}} \\
        {}\le{} & \sum_{j=0}^{g-2} (g - j) (\abs{S} - 1)^{g - j} \sum_{i=0}^j \binom{2 g}{i} q^{j -
        i/2} \overset{(ii)}{{}\le{}} 2^{2 g - 3} g^{g - 1} q^{g - 2} (\abs{S} - 1)^2.
      \end{align*}
      \qedhere
    \end{enumi}    
  \end{proof}
  
  \begin{lemma}
    Assume that $g \ge 2$ and that $q^{1/2} > \abs{S} + g$. We then have
    \begin{align*}
      & \abs{C_{g-1} - q^{g-1}} \le 2 g q^{g - 3/2} + 10 g^{g - 3} 2^{2 g - 2} q^{g-2}.
    \end{align*}
  \end{lemma}
  
  \begin{proof}
    Note that $q^{1/2} > \abs{S} + g \ge 3$, i.e. $q > 9$.
    
    First, assume that $g \ge 3$. Note that for $2 \le i \le g - 1$, \[ \binom{2 g}{i} = \binom{g -
    1}{i} \frac{(g - 1 - i)! (2 g)!}{(g - 1)! (2 g - i)!} \le \frac{(2 g)!}{(g - 1) (g - 2) (g +
    1)!} \binom{g - 1}{i}, \] whence
    \begin{align*}
      \sum_{i=1}^{g-1} \binom{2 g}{i} q^{g-1-i/2} \le{} & \frac{(2 g)!}{(g - 1) (g - 2) (g + 1)!}
      q^{g-1} \sum_{i=2}^{g-1} \binom{g - 1}{i} q^{-i/2} \\
      {}\le{} & \frac{9 (2 g)! \cdot 2^{g-1}}{g^2 (g + 1)!} q^{g-2}.
    \end{align*}
    In case $g = 2$, this term is non-negative. Thus, using Theorem~\ref{boundsonC(frakp)},
    \begin{align*}
      \abs{C_{g-1} - q^{g-1}} \le{} & 2 g q^{g - 3/2} + 2 q^{g - 2} + \sum_{i=2}^{g-1} \binom{2
      g}{i} q^{g - 1 - i/2} \\
      {}\le{} & 2 g q^{g - 3/2} + \left[ 2 + \frac{9 (2 g)! \cdot 2^{g-1}}{g^2 (g + 1)!} \right]
      q^{g-2}.
    \end{align*}
    Using $\frac{(2 g)!}{(g + 1)!} \le g^{g - 1} 2^{g - 1}$, we obtain the bound $2 g q^{g - 3/2} +
    10 g^{g - 3} 2^{2 g - 2} q^{g-2}$.
  \end{proof}
  
  \begin{proof}[Proof of Theorem~\ref{mainresult1}.]
    First, for $g = 1$ or $\abs{S} = 1$ \[ \abs{\abs{\Red_\frakp(K) \setminus \Red_S(K)} -
    (\abs{S_1} - 1)} = 0 \] by Proposition~\ref{genus1prop} in case~$g = 1$ respectively the
    definition of $\Red_S(K)$ if $\abs{S} = 1$. Hence, assume that $g > 1$ and $\abs{S} \ge 2$. We
    have
    \begin{align*}
      & \abs{\abs{\Red_\frakp(K) \setminus \Red_S(K)} - (\abs{S_1} - 1) q^{g-1}} \\
      {}\le{} & \abs{\abs{\Red_\frakp(K) \setminus \Red_S(K)} - (\abs{S_1} - 1) C_{g-1}} +
      (\abs{S_1} - 1) \abs{C_{g-1} - q^{g-1}}.
    \end{align*}
    Now, with the two lemmata, this can be bounded by
    \begin{align*}
      & 2 g (\abs{S_1} - 1) q^{g - 3/2} + 2^{2 g - 3} g^{g - 1} (\abs{S} - 1)^2 q^{g - 2} + 10 g^{g
      - 3} 2^{2 g - 2} (\abs{S_1} - 1) q^{g-2} \\
      {}\le{} & 2 g (\abs{S_1} - 1) q^{g - 3/2} + 2^{2 g} g^{g - 1} (\abs{S} - 1)^2 q^{g - 2}
    \end{align*}
    as $2^{-1} + 10 g^{-2} (\abs{S} - 1)^{-1} \le 2$.
  \end{proof}
  
  Using this theorem, we can also prove our second main result which states that the probability of
  ``stepping into a hole'', i.e. that a random element of $\Red_\frakp(K)$ lies in $\Red_S(K)$,
  equals $\frac{\abs{S} - 1}{q}$, with an error of $\calO(16^g (\abs{S} - 1) q^{-3/2})$:
  
  \begin{corollary}
    \label{mainresult2}
    Assume that $g \ge 1$ and let $S$ be as in Theorem~\ref{mainresult1}. For $q \to \infty$, we
    have \[ \abs{\frac{\abs{\Red_\frakp(K) \setminus \Red_S(K)}}{\abs{\Red_\frakp(K)}} -
    \frac{\abs{S_1} - 1}{q}} = \calO(16^g (\abs{S} - 1) q^{-3/2}). \]
  \end{corollary}
  
  \begin{proof}
    Assume that $q^{1/2} > \abs{S} + g$. Note that $\abs{\Red_\frakp(K)} \in [(\sqrt{q} - 1)^{2 g},
    (\sqrt{q} + 1)^{2 g}]$ by the Hasse-Weil bounds \cite[p.~287, Corollary~6.3 and
    Remark~6.4]{lorenzini}. Now
    \begin{align*}
      & \abs{\frac{\abs{\Red_\frakp(K) \setminus \Red_S(K)}}{\abs{\Red_\frakp(K)}} - \frac{\abs{S_1}
      - 1}{q}} \\
      {}\le{} & \abs{\frac{\abs{\Red_\frakp(K) \setminus \Red_S(K)}}{\abs{\Red_\frakp(K)}} -
      \frac{(\abs{S_1} - 1) q^{g - 1}}{\abs{\Red_\frakp(K)}}} \\
      {}+{} & \abs{\frac{(\abs{S_1} - 1) q^{g - 1}}{\abs{\Red_\frakp(K)}} - \frac{(\abs{S_1} - 1)
      q^{g - 1}}{q^g}} \\
      {}\le{} & \frac{q^g}{(\sqrt{q} - 1)^{2 g}} \left( 2 g (\abs{S_1} - 1) q^{-3/2} + 2^{2 g} g^{g
      - 1} (\abs{S} - 1)^2 q^{-2} \right) \\
      {}+{} & \frac{q^g}{(\sqrt{q} - 1)^{2 g}} (\abs{S_1} - 1) \frac{\abs{q^g -
      \abs{\Red_\frakp(K)}}}{q^{g+1}} \\
      {}\le{} & 2^{2 g + 1} g (\abs{S_1} - 1) q^{-3/2} + 2^{4 g} (\abs{S} - 1)^2 q^{-2} \\
      {}+{} & 2^{2 g} (\abs{S_1} - 1) \frac{\max\{ (1 + q^{-1/2})^{2 g} - 1, 1 - (1 - q^{-1/2})^{2
      g} \}}{q}
    \end{align*}
    using $\frac{\sqrt{q}}{\sqrt{q} - 1} < 2$ as $\sqrt{q} > 2$. Now \[ 1 - (1 - q^{-1/2})^{2 g} =
    \sum_{i=1}^{2 g} \binom{2 g}{i} (-1)^{i+1} q^{i/2} \le \sum_{i=1}^g \binom{2 g}{2 i - 1}
    q^{-1/2} < 4^g q^{-1/2} \] and, analogously, $(1 + q^{-1/2})^{2 g} - 1 < 4^g q^{-1/2}$, whence
    we obtain
    \begin{align*}
      & \abs{\frac{\abs{\Red_\frakp(K) \setminus \Red_S(K)}}{\abs{\Red_\frakp(K)}} - \frac{\abs{S_1}
      - 1}{q}} \\
      {}\le{} & 2^{2 g + 1} g (\abs{S_1} - 1) q^{-3/2} + 2^{4 g} (\abs{S} - 1)^2 q^{-2} + 2^{4 g}
      (\abs{S_1} - 1) q^{-3/2} \\
      {}\le{} & 2^{4 g} (2^{1 - 2 g} g + q^{-1/2} (\abs{S} - 1) + 1) (\abs{S} - 1) q^{-3/2} \\
      {}<{} & 2^{4 g + 2} (\abs{S} - 1) q^{-3/2}
    \end{align*}
    as $\abs{S_1} \le \abs{S}$, $2^{1 - 2 g} g < \frac{1}{2}$ and $q^{-1/2} (\abs{S} - 1) < 1$.
  \end{proof}
  
  Finally, we will give an explicit formula for $C_d(S)$ in the case that all places of $S$ have
  degree~one, i.e. $S = S_1$, and that $S \neq \emptyset$. For that, we compute the Taylor expansion
  of $h_S(t)$. For convenience, we set $\binom{-1}{0} := 1$.
  
  \begin{theorem}
    \label{mainresult3}
    Let $S = S_1$ be a finite, non-empty set of places of degree~one. Set
    \begin{itemize}
      \item $n = \abs{\{ \frakp \in S \mid \sigma(\frakp) = \frakp \}}$,
      \item $\ell = \frac{1}{2} \abs{\{ \frakp \in S \mid \frakp \neq \sigma(\frakp) \in S \}}$,
      \item $m = \abs{\{ \frakp \in S \mid \sigma(\frakp) \not\in S \}}$,
    \end{itemize}
    and let $L_K(t) = \sum_{i=0}^\infty a_i t^i$. Then
    \begin{align*}
      C_0(S) ={} & a_0 = 1, \\
      C_1(S) ={} & q - 2 \ell - m - n + 1 + a_1, \qquad \text{and for } i > 1, \\
      C_i(S) ={} & \sum_{k=0}^i \sum_{j=0}^k \sum_{p=0}^j \binom{\ell + m + n - 2 + i - k}{i - k}
      (-1)^{i - j} \binom{\ell}{k - j} q^p a_{j-p} \\
      {}-{} & \sum_{k=0}^{i-2} \sum_{j=0}^k \sum_{p=0}^j \binom{\ell + m + n - 4 + i - k}{i - k - 2}
      (-1)^{i - j} \binom{\ell}{k - j} q^{p+1} a_{j-p}.
    \end{align*}
  \end{theorem}
  
  We begin with a small lemma on the Taylor expansion on $(1 + \lambda t)^n$ with $n \in \Z$.
  
  \begin{lemma}
    \label{binomialpowerseriesexpansion}
    Let $\lambda \in \C^*$ and $n \in \N$.
    \begin{enuma}
      \item We have \[ (1 + \lambda t)^n = \sum_{i=0}^\infty \binom{n}{i} \lambda^i t^i. \]
      \item We have \[ \left( \frac{1}{1 + \lambda t} \right)^n = \sum_{i=0}^\infty \binom{i + n -
      1}{i} (-\lambda)^i t^i. \] (In case $n = 0$, we need $\binom{-1}{0} = 1$.)
    \end{enuma}
  \end{lemma}
  
  \begin{proof}
    Part (a) is clear since $\binom{n}{i} = 0$ for $i > n$. For part (b), the case $n = 0$ is clear
    since $\binom{-1}{0} = 1$ and $\binom{k - 1}{k} = 0$ for $k > 0$. For $n > 0$, we have
    \begin{align*}
      \left( \frac{1}{1 + \lambda t} \right)^n ={} & (-\lambda)^{1-n} \frac{d^{n - 1}}{d t^{n - 1}}
      \frac{1}{1 + \lambda t} = (-\lambda)^{1-n} \frac{d^{n - 1}}{d t^{n - 1}} \sum_{i=0}^\infty
      (-\lambda)^i t^i \\
      {}={} & (-\lambda)^{1 - n} \sum_{i=n - 1}^\infty \binom{i}{n - 1} (-\lambda)^i t^{i + 1 - n} =
      \sum_{i=0}^\infty \binom{i + n - 1}{n - 1} (-\lambda)^i t^i;
    \end{align*}
    finally, note that $\binom{i + n - 1}{n - 1} = \binom{i + n - 1}{i}$.
  \end{proof}
  
  \begin{proof}[Proof of Theorem~\ref{mainresult3}.]
    Note that \[ \sum_{i=0}^\infty C_i(S) t^i = \frac{(1 - t)^\ell (1 - q t^2) L_K(t)}{(1 - q t) (1
    + t)^{\ell + m + n - 1}}. \] Using the lemma, it suffices to compute the Taylor expansion of \[
    (1 - q t^2) \underbrace{\sum_{i=0}^\infty \binom{\ell}{i} (-1)^i t^i \cdot \sum_{j=0}^\infty q^j
    t^j \cdot \sum_{k=0}^\infty \binom{k + \ell + m + n - 2}{k} (-1)^k t^k \cdot \sum_{p=0}^\infty
    a_p t^p}_{=: A}, \] which can be obtained by multiplying out. First,
    \begin{align*}
      \sum_{i=0}^\infty \binom{\ell}{i} (-1)^i t^i \cdot \sum_{j=0}^\infty q^j t^j \cdot
      \sum_{p=0}^\infty a_p t^p ={} & \sum_{i=0}^\infty \binom{\ell}{i} (-1)^i t^i \cdot
      \sum_{j=0}^\infty \sum_{p=0}^j q^p a_{j-p} t^j \\
      {}={} & \sum_{i=0}^\infty \sum_{j=0}^i \sum_{p=0}^j \binom{\ell}{i - j} (-1)^{i - j} q^p
      a_{j-p} t^i.
    \end{align*}
    Using this, we obtain that $A$ equals
    \begin{align*}
      & \sum_{k=0}^\infty \binom{k + \ell + m + n - 2}{k} (-1)^k t^k \cdot \sum_{i=0}^\infty
      \sum_{j=0}^i \sum_{p=0}^j \binom{\ell}{i - j} (-1)^{i - j} q^p a_{j-p} t^i
      \\
      {}={} & \sum_{k=0}^\infty \underbrace{\sum_{i=0}^k \sum_{j=0}^i \sum_{p=0}^j \binom{\ell + m +
      n - 2 + k - i}{k - i} (-1)^{k - j} \binom{\ell}{i - j} q^p a_{j-p}}_{=: b_k} t^k.
    \end{align*}
    Note that
    \begin{align*}
      b_0 ={} & a_0 = 1 \qquad \text{and} \\
      b_1 ={} & \left( q - \ell - (\ell + m + n - 1) \right) a_0 + a_1 = q - 2 \ell - m - n + 1 +
      a_1.
    \end{align*}
    If we then multiply $A$ by $1 - q t^2$, and use these relations, we obtain the claim.
  \end{proof}
  
  Note that one can also compute the Taylor expansion by working in $\Z[[x]]/(x^{g+1})$: multiplying
  two elements requires $\calO(g^2)$ multiplications and additions in $\Z$, whence one can compute
  $h_S(t)$ from $L_K$ using a square-and-multiply method in $\calO(\log \abs{S} \cdot g^2)$
  multiplications and additions in $\Z$. Also note that the coefficients can be effectively bounded,
  using Corollary~\ref{coefficientboundcorollary} and Theorem~\ref{hasseweilbounds}.
  
  Hence, we obtained a bound on the number of hole elements as well an exact formula, as well as a
  strategy how to quickly evaluate the formula. The error terms in the bounds are by no means
  optimal, but they suffice for our needs.
  
  \section{On the Size of Holes}
  \label{onsizeofholes}
  
  Using the methods from Section~\ref{relatinginfrastructure} and Section~\ref{genfuns}, we can
  state some results on the size of holes; holes can be thought of as clusters of hole
  elements. First, we want to make this informal definition more precise.
  
  On $\Red_\frakp(K)$, define the equivalence relation \[ D \sim_S D' :\Leftrightarrow \ideal_S(D) =
  \ideal_S(D'). \] It turns out that every equivalence class contains exactly one element of
  $\Red_S(K)$. For $D \in \Red_S(K)$, we call $h(D) := [D]_{\sim_S} \setminus \{ D \}$ the
  \emph{hole} associated to $D$. Every element of $h(D)$ is a hole element, as well as any hole
  element is contained in some $h(D)$.
  
  So far, the only result known on the size of $\abs{h(D)}$ is Proposition~4.1 of
  \cite{paulus-rueck}: in the case $\abs{S} = 2$ and that the two elements in $S$ are conjugated
  under $\sigma$, and both are of degree~one, they show that $\abs{h(D)} = g - \deg D$.
  
  For all $D' \in h(D)$ we have $D \le D'$; hence, in case $\deg D = g$, $h(D) =
  \emptyset$. Assuming that $S \setminus \{ \frakp \}$ contains a place~$\frakq$ of degree~one with
  $\sigma(\frakq) \in S$, $\deg D < g$ implies $D + \frakq \in h(D)$ by
  Proposition~\ref{combinatoricaldescriptionofreds}, whence $h(D) \neq \emptyset$. We need the
  assumption that $\sigma(\frakq) \in S$, as otherwise it could happen that $\nu_{\sigma(\frakq)}(D)
  > 0$, whence $D + \frakq \ge \Conorm_{K/R}(\frakq \cap R)$.
  
  For the rest of \textbf{this section, we assume that all places in $S$ are of degree~one}.
  
  \begin{proposition}
    Let $D \in \Div_S(K)$ and set
    \begin{itemize}
      \item $n = \abs{\{ \frakp' \in S \setminus \{ \frakp \} \mid \sigma(\frakp') = \frakp' \}}$,
      \item $\ell = \frac{1}{2} \abs{\{ \frakp' \in S \setminus \{ \frakp \} \mid \frakp' \neq
      \sigma(\frakp') \in S \setminus \{ \frakp \} \}}$,
      \item $r = \abs{\{ \frakp' \in S \setminus \{ \frakp \} \mid \sigma(\frakp') = \frakp \}}$,
      \item $m = \abs{\{ \frakp' \in S \setminus \{ \frakp \} \mid \sigma(\frakp') \not\in S \}}$
      and 
      \item $m'_D = \abs{\{ \frakp' \in S \setminus \{ \frakp \} \mid \sigma(\frakp') \not\in S,
      \nu_{\sigma(\frakp')}(D) = 0 \}}$.
    \end{itemize}
    Then the following statements hold:
    \begin{enuma}
      \item We have $r, n, m, \ell, m'_D \in \N$ with $m'_D \le m$, $r \le 1$ and $\abs{S} = n + m +
      r + 2 \ell + 1$.
      \item In case $m'_D < \abs{S} - 1$, we have $h(D) = \emptyset$ if, and only if, $\deg D =
      g$. The ``if'' part also holds if $m'_D = \abs{S} - 1$.
      \item In case $\abs{S} = 1$, we have $h(D) = \emptyset$.
      \item There is a bijection between $h(D)$ and the set
      \begin{align*}
        A_{g - \deg D}^{n, \ell, m'_D + r} := \left\{ (a, b, c) \;\middle| \begin{matrix} a =
          (a_i)_i \in \{ 0, 1 \}^n, \; b = (b_i)_i \in \Z^\ell, \; c = (c_i)_i \in \N^{m'_D + r} \\
          1 \le \sum_{i=1}^n a_i + \sum_{i=1}^\ell \abs{b_i} + \sum_{i=1}^{m'_D + r} c_i \le g -
          \deg D \hfill \end{matrix} \right\}.
      \end{align*}
    \end{enuma}
  \end{proposition}
  
  \begin{proof}
    Part~(a) is clear, and part~(b) follows from the above discussion. Part~(c) is clear as in that
    case, $\Red_S(K) = \Red_\frakp(K)$. For part~(d), let
    \begin{align*}
      \{ \frakp' \in S \setminus \{ \frakp \} \mid \sigma(\frakp') = \frakp' \} ={} & \{
      \frakp_{1,1}, \dots, \frakp_{1,n}
      \} \\
      \{ \frakp' \in S \setminus \{ \frakp \} \mid \frakp' \neq \sigma(\frakp') \in S \setminus \{
      \frakp \} \} ={} & \{ \frakp_{2,1}, \dots, \frakp_{2,\ell}, \sigma(\frakp_{2,1}), \dots,
      \sigma(\frakp_{2,\ell}) \} \\
      \{ \frakp' \in S \setminus \{ \frakp \} \mid \sigma(\frakp') = \frakp \} ={} & \{
      \frakp_{3,1}, \dots, \frakp_{3,r} \} \qquad \text{and} \\
      \{ \frakp' \in S \setminus \{ \frakp \} \mid \sigma(\frakp') \not\in S,
      \nu_{\sigma(\frakp)}(D) = 0 \} ={} & \{ \frakp_{4,1}, \dots, \frakp_{4,m'_D} \}.
    \end{align*}
    Define the map
    \begin{align*}
      \Psi : A_{g - \deg D}^{n, \ell, m'_D + r} \to{} & \Div(K), \\
      ((a_i)_i, (b_i)_i, (c_i)_i) \mapsto{} & D + \sum_{i=1}^n a_i \frakp_{1,i} + \sum_{i=1}^\ell
      \max\{ b_i, 0 \} \frakp_{2,i} \\
      & \phantom{D} {}+ \sum_{i=1}^\ell \max\{ -b_i, 0 \} \sigma(\frakp_{2,i}) + \sum_{i=1}^{m'_D}
      c_i \frakp_{4,i} + \sum_{i=m'_D+1}^{m'_D+r} c_i \frakp_{3,i}.
    \end{align*}
    Clearly, $\deg \Psi((a, b, c)) \le g$, $\ideal_S(\Psi((a, b, c))) = \ideal_S(D)$ and $\Psi((a,
    b, c)) \neq D$ for all $(a, b, c) \in A_{g - \deg D}^{n, \ell, m'_D + r}$. Therefore, it
    suffices to show that the image of $\Psi$ lies in $\Red_\frakp(K)$. But this follows directly
    from the definition of $\Psi$, $A_{g - \deg D}^{n, \ell, m'_D + r}$ and
    Proposition~\ref{combinatoricaldescriptionofreds}.
  \end{proof}
  
  In the following we assume that a divisor~$D \in \Red_S(K)$ is fixed, and we use $n, m, \ell,
  m'_D$ as in the proposition.
  
  Hence, to estimate the size of $h(D)$, we have to estimate the size of $A_{g - \deg D}^{n, \ell,
  m'_D+r}$. For that, define $\hat{A}_s^{n, \ell, m'_D+r}$ as the set \[ \left\{ (a, b, c)
  \;\middle| \begin{matrix} a = (a_i)_i \in \{ 0, 1 \}^n, \; b = (b_i)_i \in \Z^\ell, \; c = (c_i)_i
    \in \N^{m'_D+r} \\ \sum_{i=1}^n a_i + \sum_{i=1}^\ell \abs{b_i} + \sum_{i=1}^{m'_D+r} c_i = s
    \hfill \end{matrix} \right\} \] for $s \in \N$; then $A_{\deg D - d}^{n, \ell, m'_D + r} =
  \bigcup_{s=1}^{\deg D - d} \hat{A}_s^{n, \ell, m'_D + r}$ is a disjoint union. Set \[ f_{n, \ell,
  m'_D + r}(t) := \sum_{s=0}^\infty \bigl|\hat{A}_s^{n, \ell, m'_D + r}\bigr| t^s; \] then, one
  quickly obtains
  \begin{align*}
    f_{n, \ell, m'_D + r}(t) ={} & \left( 1 + t \right)^n \cdot \left( 1 + 2 \sum_{s=1}^\infty t^s
    \right)^\ell \cdot \left( \sum_{s=0}^\infty t^s \right)^{m'_D + r} \\
    {}={} & (1 + t)^n \left( -1 + 2 \frac{1}{1 - t} \right)^\ell (1 - t)^{-m'_D-r} \\
    {}={} & (1 + t)^{n + \ell} (1 - t)^{-m'_D - r - \ell}.
  \end{align*}
  Using Lemma~\ref{binomialpowerseriesexpansion}, this equals (with $\binom{-1}{0} = 1$) \[
  \sum_{s=0}^\infty \underbrace{\left( \sum_{i=0}^s \binom{n + \ell}{s - i} \binom{i + m'_D + r +
  \ell - 1}{i} \right)}_{=: \hat{a}_s^{n+\ell,m'_D+r+\ell}} t^s, \] whence
  $\abs{\hat{A}_s^{n,\ell,m'_D+r}} = \hat{a}_s^{n+\ell,m'_D+r+\ell}$. Combining this with
  $\abs{h(D)} = \sum_{s=1}^{g - \deg D} \hat{a}_s^{n+\ell,m'_D+r+\ell}$, we obtain:
  
  \begin{proposition}
    \label{holesizeprop}
    We have that \[ \abs{h(D)} = \sum_{s=1}^{g - \deg D} \sum_{i=0}^s \binom{n + \ell}{s - i}
    \binom{i + m'_D + r + \ell - 1}{i}. \] \qed
  \end{proposition}
  
  This allows us to give an upper and lower bound for $\abs{h(D)}$. We begin with the upper bound.
  
  \begin{proposition}
    \label{holesizeupperbound}
    Assume that $\abs{S} \ge 2$. We then have
    \begin{align*}
      \abs{h(D)} + 1 \le{} & \binom{\abs{S} - (m - m'_D) - 1 + g - \deg D}{g - \deg D} \\
      {}\le{} & \binom{\abs{S} - 1 + g - \deg D}{g - \deg D} \le \frac{(\abs{S} - 1 + g - \deg D)^{g
      - \deg D}}{(g - \deg D)!}.
    \end{align*}
  \end{proposition}
  
  \begin{proof}
    Set $s := g - \deg D$. Clearly, one can embed $A_s^{n, \ell, m'_D + r}$ into $A_s^{0, 0, n +
    m'_D + r + 2 \ell}$ by $((a_i)_i, (b_i)_i, (c_i)_i) \mapsto ((), (), (a_1, \dots, a_n, b_1^+,
    \dots, b_{\ell}^+, b_1^-, \dots, b_{\ell}^-, c_1, \dots, c_{m'_D+r}))$ with $b_i^+ := \max\{
    b_i, 0 \}$ and $b_i^- := \max\{ -b_i, 0 \}$. Hence, we get
    \begin{equation}\tag{$\ast$}
      \abs{A_s^{n, \ell, m'_D+r}} \le \abs{A_s^{0, 0, n + m'_D + r + 2 \ell}} = \sum_{i=1}^s
      \binom{i + n + m'_D + r + 2 \ell - 1}{n + m'_D + r + 2 \ell - 1}.
    \end{equation}
    
    In case $n + m'_D + r + 2 \ell = 0$, we get $\abs{A_s^{n, \ell, m'_D + r}} = 0$. Hence, assume that
    $t := n + m'_D + r + 2 \ell - 1 \ge 0$, and $t \le \abs{S} - 2$; in that case, we get
    \begin{align*}
      \abs{A_s^{n, \ell, m'_D + r}} \le{} & \sum_{i=0}^s \binom{i + t}{i} - 1 = \binom{s + t + 1}{s}
      - 1 \\
      {}\le{} & \binom{s + \abs{S} - 1}{s} - 1 \le \frac{(\abs{S} - 1 + s)}{s!} - 1,
    \end{align*}
    which results in the claim.
  \end{proof}
  
  \begin{proposition}
    \label{holesizelowerbound}
    Assume that $\abs{S} \ge 2$. We then have \[ \abs{h(D)} + 1 \ge \sum_{i=0}^{g - \deg D}
    \binom{\abs{S} - (m - m'_D) - 2}{i}. \]
  \end{proposition}
  
  \begin{proof}
    Set $s := g - \deg D$. First, assume that $m'_D + r + \ell = 0$. Then $n + 1 = \abs{S}$, and
    \begin{align*}
      \abs{h(D)} + 1 ={} & \sum_{i=0}^s \binom{\abs{S} - 1}{i} \ge \sum_{i=0}^s \binom{\abs{S} -
      2}{i}.
    \end{align*}
    Next, assume that $m'_D + r + \ell > 0$, which implies $\abs{S} \ge 2$. We have that
    \begin{align*}
      \abs{h(D)} + 1 ={} & \sum_{i=0}^s \sum_{j=0}^i \binom{n + \ell}{i - j} \binom{j + m'_D + r +
      \ell - 1}{j} \\
      {}\ge{} & \sum_{i=0}^s \sum_{j=0}^i \binom{n + \ell}{i - j} \binom{\abs{S} - (m - m'_D) - (n +
      \ell) - 2}{j} \\
      {}={} & \sum_{i=0}^s \binom{\abs{S} - (m - m'_D) - 2}{i}
    \end{align*}
    by Vandermonde's Identity.
  \end{proof}
  
  Finally, we want to analyze the situation in the case~$\abs{S} \to \infty$ when $m = m'_D$. In
  this case, the above bounds give
  \begin{align*}
    & \frac{(\abs{S} - (m - m'_D) - 1 - (g - \deg D))^{g - \deg D}}{(g - \deg D)!} \\
    {}\le{} & \binom{\abs{S} - (m - m'_D) - 2}{g - \deg D} - 1 \le \sum_{i=0}^{g - \deg D}
    \binom{\abs{S} - 2}{i} - 1 \\
    {}\le{} & \abs{h(D)} \le \binom{\abs{S} - (m - m'_D) - 1 + g - \deg D}{g - \deg D} - 1 \\
    {}\le{} & \frac{(\abs{S} - (m - m'_D) - 1 + (g - \deg D))^{g - \deg D}}{(g - \deg D)!}.
  \end{align*}
  Hence, we obtain:
  
  \begin{corollary}
    For $\abs{S} \to \infty$, \[ \abs{h(D)} \sim \frac{(\abs{S} - (m - m'_D))^{g - \deg D}}{(g -
    \deg D)!}. \] Note that in case $(S \cup \sigma(S)) \cap \support D = \emptyset$, $m - m'_D =
    0$. \qed
  \end{corollary}
  
  Note that the set of places of $K$ of degree~one is finite; hence, $\abs{S} \to \infty$ does not
  make sense if $K$ is fixed. With $\abs{S} \to \infty$ in the corollary, we mean that for any
  sequence of hyperelliptic function fields~$K_i$, all of genus~$g$, with sets of places~$S^{(i)}$
  of degree~one of $K_i$ with $\abs{S^{(i)}} \to \infty$, and any $D_i \in \Red_{S^{(i)}}(K_i)$ with
  $\deg D_i$ independent of $i$, we have \[ \lim_{i\to\infty} \frac{h(D_i) \cdot (g - \deg
  D_i)!}{\bigl(\abs{S^{(i)}} - (m_{S^{(i)}} - m'_{S^{(i)},D})\bigr)^{g - \deg D_i}} = 1. \]
  
  Finally, for the special case~$D = 0$, the hole gets as big as it can:
  
  \begin{proposition}
    \label{zerodivisorholesize}
    Assume that $\abs{S} \ge 2$ and $S \setminus \{ \frakp \}$ contains a place lying above an
    unramified place of $R$. For $D \in \Red_S(K)$, we have that $\abs{h(D)}$ is maximal if, and
    only if, $D = 0$; in that case, \[ \abs{h(0)} = \sum_{s=1}^g \sum_{i=0}^s \binom{n + \ell}{s -
    i} \binom{i + m + r + \ell - 1}{i}. \] In particular, for $\abs{S} \to \infty$, $\abs{h(0)} \sim
    \frac{\abs{S}^g}{g!}$.
  \end{proposition}
  
  \begin{proof}
    Clearly, $D = 0$ is the only reduced divisor for which $g - \deg D = g$. For any divisor~$D \in
    \Red_S(K)$, the map $h(D) \to h(0)$, $D' \mapsto D' - D$ is injective. As there exists a
    place~$\frakp' \in S \setminus \{ \frakp \}$ which does not lie over a ramified place of $R$, $g
    \frakp' \in h(0)$ lies not in the image of the above map~$h(D) \to h(0)$ if $D \neq 0$.
  \end{proof}
  
  We have seen that the size of $h(D)$ only depends on $\deg D$, $S$ and $m_{D'}$, and were able to
  give a precise formula for $\abs{h(D)}$. Moreover, we were able to give lower and upper bounds for
  $\abs{h(D)}$ which shows the behavior for $\abs{S} \to \infty$.

  \section{Conclusion}
  \label{conclusion}
  
  We have shown that at least in the case of hyperelliptic function fields, the number of hole
  elements in the infrastructure of a global function fields behaves as expected, namely the number
  of hole elements is $n q^{g - 1}$, where $n + 1$ is the number of infinite places, up to an error
  term of $\calO(q^{g - 3/2})$. Moreover, we obtained an explicit formula for the number of hole
  elements involving only certain information on $S$ and the $L$-polynomial of $K$.
  
  A natural question is whether this holds as well for all global function fields, and if not, how
  one has to adjust the bounds. So far, the author is not aware of any answer to this
  question. Based on the results in this paper and on the experiments in \cite{landquist-thesis,
  fontein-diss}, we conjecture:
  
  \begin{conjecture}
    Let $K$ be a function field of genus~$g$ with exact constant field~$\F_q$, and assume that $S$
    is a finite set of places of $K$ containing at least one place of degree~one. Then, for $q \to
    \infty$, \[ \abs{\abs{\Red_\frakp(K) \setminus \Red_S(K)} - (\abs{S_1} - 1)} = 2 g (\abs{S_1} -
    1) q^{g - 3/2} + \calO(q^{g - 2}) \] and \[ \abs{\frac{\abs{\Red_\frakp(K) \setminus
    \Red_S(K)}}{\abs{\Red_\frakp(K)}} - \frac{\abs{S_1} - 1}{q}} = \calO(q^{-3/2}), \] where the
    $\calO$-constants only depend on $\abs{S}$ and $g$.
  \end{conjecture}
  
  Moreover, we have shown that the size of a hole next to a reduced divisor~$D \in \Red_S(K)$
  depends highly on $g - \deg D \in \{ 0, \dots, g \}$: in case $g = \deg D$, the size is
  zero. Otherwise, if~$S = S_1$, there usually exists at least one hole element next to
  $D$. Assuming that $\sigma(S)$ meets $\support D$ in $m_S$ places, the size of the hole next to
  $D$ behaves like $\frac{(\abs{S} - m_S)^{g - \deg D}}{(g - \deg D)!}$ for $\abs{S} \to \infty$.
  
  We conjecture that this holds in a similar way for all function fields:
  
  \begin{conjecture}
    Let $K$ be a function field of genus~$g$ with exact constant field~$\F_q$, and assume that $S$
    is a non-empty finite set of places of $K$, all of degree~one. Then there exists another finite
    set $S'$ with $S \subseteq S'$ of size $\abs{S'} = \calO(\abs{S})$ such that for all $D \in
    \Red_{S'}(K)$, we asymptotically have \[ \abs{h(D)} \sim \frac{\abs{S}^{g - \deg D}}{(g - \deg
    D)!} \] as $\abs{S} \to \infty$ with $S' \cap \support D = \emptyset$. Moreover, $\abs{h(D)}$
    only depends on $\deg D$ (and of course $S$ and $K$) as long as $S' \cap \support D =
    \emptyset$.
  \end{conjecture}
  
  In the case of hyperelliptic function fields with quadratic rational subfield~$R$ and the unique
  non-trivial $R$-automorphism $\sigma$ of $K$, we can set $S' = S \cup \sigma(S)$ and obtain
  $\abs{S'} \le 2 \abs{S}$. We moreover assume that a result similar to
  Proposition~\ref{zerodivisorholesize} holds in general:
  
  \begin{conjecture}
    Let $K$ be a function field of genus~$g$ with exact constant field~$\F_q$, and assume that $S$
    is a non-empty finite set of places of $K$, all of degree~one, with at least one place $\neq
    \frakp$ which lies outside a finite set only dependent of $K$. Then $\abs{h(D)}$ is maximal for
    $D \in \Red_S(K)$ if, and only if, $D = 0$.
  \end{conjecture}
  
  In the hyperelliptic case, the finite set of places of $K$ which have to be avoided are the places
  lying above places of $R$ which ramify in $K$.
  
%
  
\newcommand{\etalchar}[1]{$^{#1}$}

  
\end{document}